\documentclass[reqno, oneside]{amsart}
\usepackage{amsfonts, amssymb, amsmath, color, eucal, amsthm}

\usepackage{tikz,tikz-cd}
\usepackage[all]{xy}

\usepackage{color}
\definecolor{myurlcolor}{rgb}{0.1,0.1,0.8}
\definecolor{mylinkcolor}{rgb}{0.05,0.05,0.4}

\usepackage{hyperref}
\hypersetup{colorlinks,
	linkcolor=mylinkcolor,
	citecolor=mylinkcolor,
	urlcolor=myurlcolor}

\usepackage{cleveref}

\newtheorem{lemma}{Lemma}
\numberwithin{lemma}{section}
\newtheorem{proposition}[lemma]{Proposition}
\newtheorem{theorem}[lemma]{Theorem}

\theoremstyle{definition}
\newtheorem{definition}[lemma]{Definition}
\newtheorem{example}[lemma]{Example}
\newtheorem{remark}[lemma]{Remark}

  \newlength\squareheight
  \newcommand\squareslash{\tikz{\draw (0,0) rectangle (\squareheight,\squareheight);\draw(0,0) -- (\squareheight,\squareheight)}}
  \DeclareMathOperator\squarediv{\squareslash}

  \newlength\subsquareheight
\newcommand{\demph}[1]{\emph{#1}}
\newcommand{\xto}{\xrightarrow}
\newcommand{\bpc}{\mathrm{RC}}
\newcommand{\bph}{\mathrm{RH}}
\newcommand{\Pre}{\mathrm{Pre}}

\newcommand{\Nv}{\mathcal{N}}
\newcommand{\calF}{\mathcal{F}}
\newcommand{\calG}{\mathcal{G}}

\newcommand{\cn}[1]{\mathbf{#1}}
\newcommand{\cs}[1]{\mathbb{#1}}
\DeclareMathOperator{\sign}{sign}

\title[The reachability homology of a directed graph]{The reachability homology of a directed graph}

\author{Richard Hepworth}
\address{Institute of Mathematics\\ University of Aberdeen}
\email{r.hepworth@abdn.ac.uk}

\author{Emily Roff}
\address{Department of Mathematics\\ Osaka University}
\email{emily.roff@ed.ac.uk}

\subjclass[2020]{
Primary 
18G90, 
05C25; 
Secondary 
05C31, 
05C38, 
18G35
}

\keywords{}

\begin{document}

\begin{abstract}
    The last decade has seen the development of path homology
    and magnitude homology---two homology theories of directed
    graphs, each satisfying classic properties such 
    as K\"unneth and Mayer--Vietoris theorems.
    Recent work of Asao has shown that magnitude homology and path 
    homology are related, appearing in different
    pages of a certain spectral sequence.
    Here we study the target of that spectral sequence,
    which we call \emph{reachability homology}. We prove that it satisfies appropriate
    homotopy invariance, K\"unneth, excision, and Mayer--Vietoris
    theorems, these all being stronger than the corresponding 
    properties for either magnitude or path homology.
\end{abstract}

\maketitle



\section{Introduction}\label{sec:intro}

Recent years have seen the emergence of several novel homological invariants of directed graphs, and associated developments in the realm of discrete homotopy theory.
Among those invariants, \emph{path homology} and \emph{magnitude homology} in particular have each amassed a sizeable literature. 
Despite their separate evolution, these two theories were recently shown to be related 
by the so-called \emph{magnitude-path spectral sequence}.
This paper studies the {target} of that 
spectral sequence, which we call \emph{reachability homology},
and establishes its homological properties.
As well as introducing a further homology theory for graphs, 
this work advances the
study of the global structure and properties
of the magnitude-path spectral sequence.

The path homology of directed graphs was first defined by Grigor'yan, Lin, Muranov and Yau in~\cite{GLMY2013}, with cohomological precursors due to Dimakis and M\"uller-Hoissen in~\cite{DMH1,DMH2}. Path homology has been shown to satisfy a form of homotopy invariance \cite{PH-homotopy} and analogues of the other Eilenberg--Steenrod axioms \cite{PH-eilenbergsteenrod}, as well as K\"unneth theorems with respect to the box product and the join of directed graphs \cite{PH-kunneth}. A version of discrete Morse theory relevant to path homology has been developed in~\cite{PH-discretemorse}, and interesting analogues have been obtained, for graphs, of classical geometric results \cite{PH-curvature}.

Magnitude homology, meanwhile, was originally constructed to categorify a numerical invariant of graphs known as \emph{magnitude} \cite{LeinsterGraph}, in a manner analogous to the categorification of the Jones polynomial by Khovanov homology. It was first defined for (undirected) graphs by Hepworth and Willerton in \cite{HepworthWillerton2017}, where its homological properties were established. The magnitude homology of a graph is a stronger invariant than magnitude \cite{Gu}; it can contain torsion \cite{KanetaYoshinaga}; and it has been shown to capture geometric features including information about girth \cite{AsaoGirth}. The construction has been extended by Leinster and Shulman to encompass all metric spaces and a large class of enriched categories \cite{LeinsterShulman}. In the context of metric spaces it has been shown to capture information concerning curvature \cite{Asao-curvature}, convexity \cite{LeinsterShulman}, and the uniqueness of geodesics \cite{GomiGeodesic}; connections have also been established with persistent homology \cite{Otter, GovcHepworth}. Most recently, magnitude homology has been extended to certain higher categories and metric groups \cite{Roff2023}, and an associated `magnitude homotopy type' for metric spaces has been constructed \cite{TajimaYoshinaga}.

As this collection of references suggests, the theories of path homology and of magnitude homology have developed quite independently. Despite this, Asao demonstrated in 2022 that the two are in fact intimately related \cite{Asao-path}. To any directed graph one can associate a spectral sequence---the \emph{magnitude-path spectral sequence} or \emph{MPSS}---whose \(E^1\) page is exactly magnitude homology, while path homology can be identified with a single axis of page \(E^2\).

The subject of the present paper is the target object of the MPSS, which we name \emph{reachability homology}. 
It is the homology of the chain complex whose generators in degree~$k$
are tuples $(x_0,\ldots,x_k)$ of vertices with the property that there
is a directed path from each entry in the tuple to the next.
In other words, each entry in the tuple can \emph{reach} the next one. The reachability homology of a directed graph \(G\) is thus an invariant of the \emph{reachability relation} of \(G\): the preorder on the set of vertices in which \(u \leq v\) if there exists a directed path in \(G\) from \(u\) to \(v\). From any preorder one can construct a simplicial set known as the \emph{nerve}. The reachability homology of a directed graph is precisely the homology of the normalized complex of chains in the nerve of its reachability relation. (For details, see \Cref{sec:RH}.)

Determining the reachability relation of a directed graph allows one to determine its strongly connected components; algorithms to achieve this have been the subject of research in computer science and combinatorics since the 1970s, and remain so today \cite{AjtaiFagin, Kameda, Tarjan, TarjanZwick}. In a more algebraic direction, Caputi and Riihim\"aki have recently studied the reachability preorder in connection with path categories and path algebras of quivers \cite{CaputiRiihimaki}. As well as contributing to the developing understanding of the magnitude-path spectral sequence, the present paper offers a novel homological perspective that may eventually be brought to bear on practical questions of reachability and strong connectivity.

Our main results establish that reachability homology indeed deserves to be called
a homology theory, possessing the following characteristic 
properties:
\begin{itemize}
    \item
    It satisfies homotopy invariance with 
    respect to a relation on graph maps that we call \emph{long homotopy}.
    (See~\Cref{thm:long_htpy_inv}.)
    \item
    It satisfies a K\"unneth theorem with respect to both the
    cartesian product (the strong product) and the box product.
    (See~\Cref{thm:RH_kunneth}.)
    \item
    It satisfies excision and Mayer-Vietoris theorems
    for pushouts along a class of maps that we call
    \emph{long cofibrations}.
    (See~\Cref{thm:excision_MV_reach}.)
\end{itemize}
Thus, magnitude homology, path homology \emph{and} reachability homology
all deserve---by their properties---to be regarded as homology theories
of graphs.
However, these are certainly different theories
with different properties, 
even if we give some of these properties the same generic
names (such as K\"unneth or excision theorems).

To illustrate the differences, 
consider the following three graphs.
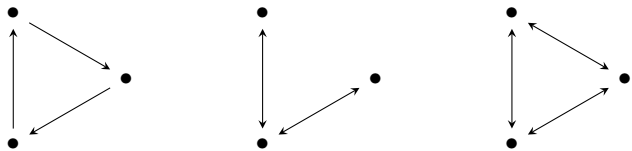
\[
    \begin{tikzpicture}
        \node (a) at (0:1) {$\bullet$};
        \node (c) at (120:1) {$\bullet$};
        \node (e) at (240:1) {$\bullet$};
    
        \draw[stealth-] (e) -- (a);
        \draw[stealth-] (a) -- (c);
        \draw[stealth-] (c) -- (e);
    
    \end{tikzpicture}
    \qquad\qquad
    \begin{tikzpicture}
        \node (a) at (0:1) {$\bullet$};
        \node (c) at (120:1) {$\bullet$};
        \node (e) at (240:1) {$\bullet$};
    
        \draw[stealth-stealth] (e) -- (a);
        \draw[stealth-stealth] (c) -- (e);
    
    \end{tikzpicture}
    \qquad\qquad
    \begin{tikzpicture}
        \node (a) at (0:1) {$\bullet$};
        \node (c) at (120:1) {$\bullet$};
        \node (e) at (240:1) {$\bullet$};
    
        \draw[stealth-stealth] (e) -- (a);
        \draw[stealth-stealth] (a) -- (c);
        \draw[stealth-stealth] (c) -- (e);
    
    \end{tikzpicture}
\]
Magnitude homology can distinguish all three graphs, for example because they each have different numbers of directed edges. Path homology can distinguish the first graph from the other two, but cannot distinguish the second and third from one another. (The second and third graphs are homotopy equivalent to a point in the sense appropriate to path homology, whereas the path homology of the first graph is nontrivial in degree $1$.) Reachability homology, meanwhile, identifies all three graphs, 
because all three are long homotopy equivalent to a point.

The magnitude-path spectral sequence offers a systematic account of these different points of view. Asao has shown that the $E^{r+1}$ page of the MPSS has a homotopy invariance property that holds when maps of directed graphs $f,g\colon G\to H$ are \emph{$r$-close} with respect to the shortest path metric~\cite[Section~4]{Asao-filtered}. In \cite{HepworthRoff2024} the present authors demonstrate that each page of the sequence possesses homological properties compatible with this ever-stronger homotopy invariance. Thus, the MPSS as a whole encompasses a \emph{spectrum} of homological perspectives on directed graphs, interpolating between magnitude homology and reachability homology.

Reachability homology, as the target object, controls the eventual behaviour of the MPSS, and has stronger homological properties than any of the individual pages. Meanwhile, its elementary construction aids in establishing those properties by making transparent its relationship to more classical homotopy-theoretic notions. (This is exemplified here by our proof of the excision theorem, which is made easy by exploiting a well known result in the homotopy theory of small categories.) Reachability homology therefore functions as a touchstone for the magnitude-path spectral sequence and, we hope, for those who study its pages.

\subsection*{Structure of the paper}

\Cref{sec:digraph_preord} describes basic properties of the reachability relation that will be relevant to our story. In \Cref{sec:nerve} we describe the nerve of a directed graph, which is used in \Cref{sec:RH} to construct the reachability complex and reachability homology. In \Cref{sec:RH_kunneth} we prove the K\"unneth theorem for reachability homology, and in \Cref{sec:RH_MV} we prove the excision theorem and associated Mayer--Vietoris theorem.

\subsection*{Acknowledgements}

We are grateful to Masahiko Yoshinaga and Sergei O.~Ivanov for helpful remarks. This work was partially supported by
JSPS Postdoctoral Fellowships for Research in Japan.


\section{The reachability relation}\label{sec:digraph_preord}

The topic of this paper is a homology theory for directed graphs that arises by regarding a graph as a preorder via the \emph{reachability relation} on its vertices. This section describes the reachability relation and outlines relationships between the categories of directed graphs and preorders that will be important in what follows. Many different categories of directed graphs appear in the literature, so we begin by specifying our definition. The choices made here will be accounted for in \Cref{rmk:digraph_def}. 

\begin{definition}\label{def:digraph}
    A \demph{directed graph} \(G\) consists of a set of \demph{vertices} \(V(G)\) and a set of \demph{directed edges} \(E(G) \subseteq V(G) \times V(G)\). A directed edge \((u,v)\) is depicted by an arrow \(u \to v\).
    A \demph{map of directed graphs} \(G \to H\) is a function \(f: V(G) \to V(H)\) with the property that for every directed edge \(u \to v\) in \(G\), either  \(f(u) = f(v)\) or there is a directed edge \(f(u) \to f(v)\) in \(H\) (or both). We denote the category of directed graphs and maps of directed graphs by \(\cn{DiGraph}\). 
\end{definition} 

\begin{definition}\label{def:directed_path}
    A \demph{directed path} in a directed graph \(G\) is a tuple \((v_0, \ldots, v_k)\) of vertices in \(G\) such that for each \(0 \leq i < k\) either \(v_i = v_{i+1}\) or \((v_i, v_{i+1})\) is a directed edge in \(G\) (or both).
\end{definition}

Notice that we do not require the vertices in a directed path to be distinct.

Recall that a \emph{preorder} $\leq$ on a set $X$ 
is a binary relation on $X$
that is reflexive and transitive, and that a 
\emph{monotone} map between preordered sets is one that 
preserves the preorders.
Preorders and monotone maps form a category $\cn{PreOrd}$.

\begin{definition}\label{def:reachability}
    Let \(G\) be a directed graph. The \demph{reachability relation} of \(G\) is the preorder \(\Pre(G)\) on the vertices of \(G\) in which \(u \leq v\) if and only if there exists a directed path from \(u\) to \(v\) in \(G\).
    The operation of assigning to a directed graph its reachability
    relation extends to a functor
    \[\Pre: \cn{DiGraph} \to \cn{PreOrd}.\]
\end{definition}

A directed graph \(G\) is the same thing as a binary relation on the set \(V(G)\). 
Thus there is an inclusion 
\[\iota: \cn{PreOrd} \hookrightarrow \cn{DiGraph}\]
that
takes a preorder \(P\) to the directed graph in which there is an edge \(p \to q\) whenever \(p \leq q\) in \(P\). Indeed, this is a \emph{reflective} subcategory: the functor \(\iota\) has a left adjoint, 
namely $\Pre$ itself.

\begin{lemma}\label{lem:Pre_adjoint}
    The functor $\Pre$ is left-adjoint to the functor $\iota$.
\end{lemma}

\begin{proof}
Let \(G\) and \(H\) be directed graphs.  To establish the adjunction it suffices to show that, given a directed graph \(G\) and a preorder \(P\), a function \(f\) from the vertices of \(G\) to the elements of \(P\) determines a monotone map \(\Pre(G) \to P\) if and only if it determines a map of directed graphs \(G \to \iota(P)\).

Indeed, for any preorder \(P\) we have \(\Pre(\iota(P)) = P\), so if \(f\) determines a map of directed graphs \(G \to \iota(P)\) then functoriality ensures it determines a monotone map \(\Pre(G) \to \Pre(\iota(P)) = P\). For the converse, assume \(f\) is monotone, and take any edge \(u \to v\) in \(G\). By definition of \(\Pre(G)\) we have \(u \leq v\), so the monotonicity of \(f\) implies \(f(u) \leq f(v)\) in $P$, i.e.~there is an edge \(f(u) \to f(v)\) in \(\iota(P)\). Thus, \(f\) determines a map of directed graphs \(G \to \iota(P)\).
\end{proof}

\begin{remark}\label{rmk:digraph_def}
    Our definition of a directed graph allows loops, but does not demand them; on the other hand, it does not allow multiple edges. We \emph{could} allow multiple edges without substantially changing the story of this paper: any parallel edges---as well as the presence or absence of loops---are forgotten upon taking the reachability relation of a directed multigraph. Since the constructions that concern us factor through the reachability relation, they are insensitive to loops or multiple edges. Thus, the results in this paper can readily be extended to the setting of quivers, whose reachability relations are studied by Caputi and Riihim\"aki in \cite{CaputiRiihimaki}.
    
    By contrast, the choice of morphisms is significant. Whereas some authors ask that morphisms of directed graphs preserve all edges, our morphisms are permitted to preserve edges or contract them. This choice does affect the development; in particular, it gives us Lemma~\ref{lem:Pre_adjoint}, which will be important in the proof of the excision theorem for reachability homology (\Cref{thm:excision_nerve}).
\end{remark}

\begin{remark}\label{rmk:trans_clos}
The reachability relation on a directed graph \(G\) is the reflexive and transitive closure of the binary relation represented by the edges of \(G\). Taking the reflexive transitive closure of a relation is an idempotent operation; consequently, \(\Pre\) does not distinguish \(G\) from the directed graph representing its reflexive, transitive closure. More generally, one can freely add edges between any vertices in \(G\) that are already connected by a path, without changing the reachability relation. 
\end{remark}

\begin{example}\label{eg:triangles}
    Consider again these three directed graphs.
    
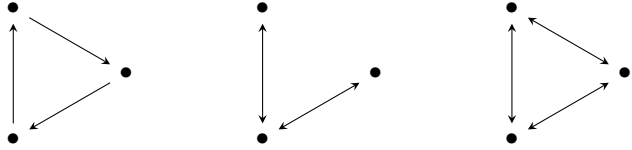
\[
        \begin{tikzpicture}
            \node (a) at (0:1) {$\bullet$};
            \node (c) at (120:1) {$\bullet$};
            \node (e) at (240:1) {$\bullet$};
    
            \draw[stealth-] (e) -- (a);
            \draw[stealth-] (a) -- (c);
            \draw[stealth-] (c) -- (e);
    
        \end{tikzpicture}
        \qquad\qquad
        \begin{tikzpicture}
            \node (a) at (0:1) {$\bullet$};
            \node (c) at (120:1) {$\bullet$};
            \node (e) at (240:1) {$\bullet$};
    
            \draw[stealth-stealth] (e) -- (a);
            \draw[stealth-stealth] (c) -- (e);
    
        \end{tikzpicture}
        \qquad\qquad
        \begin{tikzpicture}
            \node (a) at (0:1) {$\bullet$};
            \node (c) at (120:1) {$\bullet$};
            \node (e) at (240:1) {$\bullet$};
    
            \draw[stealth-stealth] (e) -- (a);
            \draw[stealth-stealth] (a) -- (c);
            \draw[stealth-stealth] (c) -- (e);
    
        \end{tikzpicture}
    \]
    In each of these one can travel from any vertex to any other
    along a directed path.
    It follows that all three graphs 
    have the same reachability relation, 
    namely the `complete' preorder in which 
    $v\leq w$ for any two vertices $v$ and $w$.
\end{example}

As \Cref{rmk:trans_clos} and \Cref{eg:triangles} illustrate, taking the reachability relation of a directed graph destroys a great deal of information. To retain more information, one might instead choose to record not just the existence of paths between vertices, but their lengths---regarding a directed graph not as a preorder but as a metric space. Here, the term `metric space' refers to what is sometimes called an `extended quasi-metric space' or just a `generalized metric space': a set equipped with a distance function that satisfies the triangle inequality and is zero on the diagonal, but need not be symmetric and may take infinite values \cite[Example 1.2.2 (iii)]{LeinsterMetric2013}.
Every directed graph \(G\) determines a metric \(d_G\) on its set of vertices, the \demph{shortest path metric}, in which \(d_G(u,v)\) is the minimal number of edges in a directed path from \(u\) to \(v\) in \(G\), or infinity if no such path exists.

In \Cref{sec:RH_kunneth} we will consider two binary operations on the category of directed graphs: the \emph{box product} and the \emph{strong product}.

\begin{definition}\label{def:box_and_strong}
The \demph{box product} of directed graphs \(G\) and \(H\) is the directed graph \(G \square H\) with \(V(G \square H) = V(G) \times V(H)\) and \(((g_1,h_1),(g_2,h_2)) \in E(G \square H)\) if
\begin{itemize}
    \item \(g_1 = g_2\) and \((h_1,h_2) \in E(H)\), or
    \item \((g_1,g_2) \in E(G)\) and \(h_1 = h_2\).
\end{itemize}
\end{definition}

The box product is sometimes referred to as the `cartesian product' of directed graphs; however, it is not the categorical product. Rather, when \(\cn{DiGraph}\) is defined as in \Cref{def:digraph}, that distinction belongs to the strong product.

\begin{definition}\label{def:strong}
The \demph{strong product} of directed graphs \(G\) and \(H\) is the directed graph \(G \squarediv H\) with \(V(G \squarediv H) = V(G) \times V(H)\) and \(((g_1,h_1), (g_2,h_2)) \in E(G \squarediv H)\) if
\begin{itemize}
\item \(g_1 = g_2\) and \((h_1, h_2) \in E(H)\), or
\item \((g_1,g_2) \in E(G)\) and \(h_1 = h_2\), or
\item \((g_1,g_2) \in E(G)\) and \((h_1,h_2) \in E(H)\).
\end{itemize}
\end{definition}

\begin{remark}
    Though the box product is not the categorical product, it does acquire a very natural categorical interpretation when viewing directed graphs as metric spaces via the shortest path metric. For directed graphs \(G\) and \(H\), one has
    \[d_{G \square H}((g,h),(g',h')) = d_G(g,g') + d_H(h,h')\]
    for all \((g,h), (g',h') \in V(G) \times V(H)\); in other words, the box product, under the shortest path metric, is the \(\ell_1\)-product of the metric spaces \((G,d_G)\) and \((H,d_H)\). The \(\ell_1\)-product is the tensor product for the monoidal structure that arises naturally when considering a metric space as an enriched category. (For the classical account of this perspective on metric spaces, see \cite{LawvereMetric1974}.) It is this which accounts for the importance of the box product in the theory of magnitude homology (see \cite[Propositions 1.4.3 and 2.3.6]{LeinsterMetric2013} and \cite[Section 4]{Roff2023}).
\end{remark}

Though the strong product and the box product look quite different as directed graphs and play quite different roles structurally in \(\cn{DiGraph}\), from the perspective of reachability they are indistinguishable, as the next lemma shows.

\begin{lemma}\label{lem:box_v_strong}
    Let \(G\) and \(H\) be directed graphs, and let \((g,h), (g', h')\) be elements of \(V(G) \times V(H)\). The following are equivalent:
    \begin{enumerate}
        \item There exists a directed path from \((g,h)\) to \((g',h')\) in \(G \square H\). \label{eq:box_v_strong1}
        \item There exists a directed path from \((g,h)\) to \((g',h')\) in \(G \squarediv H\). \label{eq:box_v_strong2}
        \item There exist directed paths from \(g\) to \(g'\) in \(G\) and from \(h\) to \(h'\) in \(H\). \label{eq:box_v_strong3}
    \end{enumerate}
\end{lemma}

\begin{proof}
    Since every directed edge in \(G \square H\) is also a directed edge in \(G \squarediv H\), it is clear that (\ref{eq:box_v_strong1}) implies (\ref{eq:box_v_strong2}).
    
    To see that (\ref{eq:box_v_strong2}) implies (\ref{eq:box_v_strong3}), take any directed path from \((g, h)\) to \((g', h')\) in \(G \squarediv H\)---say, \(((g_0, h_0), \ldots, (g_n,h_n))\). Consider the tuple \((g_0, \ldots, g_n)\) of vertices in \(G\). Since for each \(0 \leq i < n\) there is an edge \((g_i,h_i) \to (g_{i+1}, h_{i+1})\) in \(G \squarediv H\), for each \(i\) we must have either \(g_i=g_{i+1}\) or else an edge \(g_i \to g_{i+1}\) in \(G\). In other words, \((g_0, \ldots, g_n)\) is a directed path from \(g\) to \(g'\) in \(G\). Similarly, \((h_0, \ldots, h_n)\) is a directed path from \(h\) to \(h'\) in \(H\).
    
    Now suppose \((g_0, \ldots, g_m)\) is a directed path in \(G\) from \(g = g_0\) to \(g' = g_m\), and \((h_0, \ldots, h_n)\) is a path in \(H\) from \(h = h_0\) to \(h' = h_n\). Then 
    \[\left((g_0, h_0), \ldots, (g_m, h_0), (g_m,h_1), \ldots, (g_m, h_n)\right)\]
    is a directed path in \(G \square H\) from \((g,h)\) to \((g',h')\); this says (\ref{eq:box_v_strong3}) implies (\ref{eq:box_v_strong1}).
\end{proof}

The product of preorders \(P\) and \(Q\), denoted \(P \land Q\), is the set \(P \times Q\) equipped with the preorder in which \((u_1, v_1) \leq (u_2, v_2)\) if and only if \(u_1 \leq u_2\) in \(P\) and \(v_1 \leq v_2\) in \(Q\). In light of \Cref{lem:box_v_strong}, the following proposition is immediate.

\begin{proposition}\label{prop:pre_digraph_monoidal}
Let \(G\) and \(H\) be directed graphs. Then
\[\Pre(G \square H) \cong \Pre(G \squarediv H) \cong \Pre(G) \land \Pre(H),\]
naturally in \(G\) and \(H\). \qed
\end{proposition}


\section{The nerve of a directed graph}\label{sec:nerve}

We have seen that every directed graph determines a preorder on its set of vertices. In turn, every preordered set \((X, \leq)\) can be regarded as a small category with objects the elements of \(X\) and an arrow \(u \to v\) if and only if \(u \leq v\). A function between preorders is monotone if and only if it defines a functor between the corresponding categories, so \(\cn{PreOrd}\) embeds as a full subcategory of \(\cn{Cat}\), the category of small categories. We will routinely regard preorders as categories via this embedding.

In this setting, natural transformations have an especially simple description.

\begin{lemma}\label{lem:preord_nt}
    Let \((X,\leq)\) and \((Y, \leq)\) be preordered sets, and \(f,g: X \to Y\) be monotone maps (equivalently, functors). There is a natural transformation \(f \Rightarrow g\) if and only if \(f(x) \leq g(x)\) for every \(x\in X\).
\end{lemma}

\begin{proof}
    The data of a natural transformation \(f \Rightarrow g\) consists precisely of an arrow \(f(x) \to g(x)\) for each \(x \in X\), which is to say an inequality \(f(x) \leq g(x)\). The naturality condition is automatically satisfied since, in a preorder, for every pair of composable arrows there is exactly one possible composite, which guarantees that every square commutes.
\end{proof}

We follow Di \textit{et al} \cite{DIMZ} in defining the \emph{nerve} of a directed graph \(G\) to be the nerve of \(\Pre(G)\), regarded as a category. 
The nerve of a small category \(\cs{A}\) is a simplicial set whose topology captures information about the structure of \(\cs{A}\);  see~\cite{Segal} or~\cite[\S XII.2]{MacLane}.
The nerve construction is functorial: every functor between small categories induces a simplicial map between their nerves. Moreover, given functors \(F,G: \cs{A} \to \cs{B}\), any natural transformation \(\alpha: F \Rightarrow G\) gives rise to a homotopy, in the sense of simplicial sets, between the maps induced by \(F\) and \(G\)
(see~\cite[Proposition~2.1]{Segal}, where the result is proved for
classifying spaces; the same proof applies to nerves).
See~\cite[\S I]{Segal} or~\cite[\S VII.5]{MacLane} for brief introductions to the theory of simplicial sets.

In the case of a preorder \(P\), the nerve \(\Nv(P)\) is the simplicial set whose \(k\)-simplices are chains \((p_0 \leq p_1 \leq \cdots \leq p_k)\) of elements in \(P\). For \(0 \leq i \leq k\) the face operator \(\delta_i\) discards the \(i^\mathrm{th}\) term in each chain---
\[\delta_i (p_0 \leq \cdots \leq p_k) = (p_0 \leq \cdots \leq \widehat{p_i} \leq \cdots \leq p_k)\]
---while the \(i^\mathrm{th}\) degeneracy operator duplicates the \(i^\mathrm{th}\) term:
\[\sigma_i (p_0 \leq \cdots \leq p_k) = (p_0 \leq \cdots \leq p_i \leq p_i \leq \cdots \leq p_k).\]
Let us spell out what this gives for the reachability relation of a directed graph. We abuse notation slightly by writing \(\Nv(G)\) to denote the simplicial set \(\Nv(\Pre(G))\).

\begin{definition}\label{def:digraph_nerve}
The \demph{nerve} of a directed graph \(G\) is the simplicial set \(\Nv(G)\) whose set of \(k\)-simplices is
\[\Nv_k(G) = \left\{(v_0, \ldots, v_k) \middle\vert 
\begin{array}{l}
v_0 \ldots v_k \in V(G) \text{ and for } 0 \leq i < k \text{ there} \\
\text{exists a directed path from } v_i \text{ to } v_{i+1} \text{ in } G
\end{array}
\right\}.\]
For \(0 \leq i \leq k\) the face operator \(\delta_i: \Nv_k(G) \to \Nv_{k-1}(G)\) discards the \(i^\mathrm{th}\) vertex:
\[\delta_i(v_0, \ldots, v_k) = (v_0, \ldots, \widehat{v_{i}}, \ldots, v_k).\]
(The face operators are well defined since a directed path from \(x_{i-1}\) to \(x_i\) can be concatenated with one from \(x_i\) to \(x_{i+1}\) to obtain a directed path from \(x_{i-1}\) to \(x_{i+1}\).)
The degeneracy operator \(\sigma_i: \Nv_k(G) \to \Nv_{k+1}(G)\) duplicates the \(i^\mathrm{th}\) vertex:
\[\sigma_i(v_0, \ldots, v_k) = (v_0, \ldots, v_i, v_i, \ldots, v_k).\]
\end{definition}

\begin{proposition}\label{prop:nerve_maps}
    Every map of directed graphs \(f: G \to H\) induces a simplicial map \(\Nv(f): \Nv(G) \to \Nv(H)\) specified on simplices by
    \[\Nv(f)(v_0,\ldots,v_k) = (f(v_0), \ldots, f(v_k)).\]
\end{proposition}

\begin{proof}
    The construction \(\Nv\) can be obtained as the composite
    \begin{equation}\label{eq:nerve_maps}
        \cn{DiGraph} \xto{\Pre} \cn{PreOrd} \hookrightarrow \cn{Cat} \xto{\Nv} \cn{SSet},
    \end{equation}
    where \(\cn{SSet}\) is the category of simplicial sets and simplicial maps. Here, the first arrow takes a graph to its reachability relation; it is functorial by \Cref{lem:Pre_adjoint}. The second arrow is the functor embedding \(\cn{PreOrd}\) into \(\cn{Cat}\), and the third is the functor taking the nerve of small category. It follows that \(\Nv\) is functorial, and the description of the induced map is obtained by following \(f\) through (\ref{eq:nerve_maps}).
\end{proof}

The nerve construction also supplies a natural notion of homotopy between maps of directed graphs.

\begin{definition}\label{def:long_homotopy}
Given maps of directed graphs $f,g\colon G\to H$, we say \emph{there is a long homotopy from $f$ to $g$}, and write $f\rightsquigarrow g$, if, for each vertex $v$ of $G$, there is a directed path in $H$ from $f(v)$ to $g(v)$.
\end{definition}

Observe that the relation \(f\rightsquigarrow g\) is a \emph{condition} on $f$ and $g$ that requires the existence of certain paths in $H$. But it does not require us to make any choices of such paths, let alone a coherent such choice. 
Note that the relation $\rightsquigarrow$ is transitive but not necessarily symmetric.

\begin{definition}
    We say the directed graphs $G$ and $H$ are \emph{long-homotopy equivalent} if there are maps $f\colon G\to H$ and $g\colon H\to G$ such that $f\circ g$ is related to $\mathrm{Id}_H$ by a zig-zag of long
    homotopies and $g\circ f$ is related to $\mathrm{Id}_G$
    by a zig-zag of long homotopies.
    In this situation we call $f$ and $g$ \emph{long-homotopy equivalences}.
\end{definition}

\begin{proposition}\label{prop:long_homotopy}
    Let \(G\) and \(H\) be directed graphs, and \(f,g: G \to H\) be maps of directed graphs. If \(f \rightsquigarrow g\) then there is a homotopy from \(\Nv(f)\) to \(\Nv(g)\).
\end{proposition}

\begin{proof}
Suppose \(f \rightsquigarrow g\); that is, for every \(v \in V(G)\) there is a directed path from \(f(v)\) to \(g(v)\) in \(H\). Passing to the reachability relation of both graphs, this says that \(f(v) \leq g(v)\) for every \(v \in V(G)\), which, by \Cref{lem:preord_nt}, is equivalent to the existence of a natural transformation \(f \Rightarrow g\). Since natural transformations induce homotopies on taking the nerve, the statement follows.
\end{proof}


\section{Reachability homology}\label{sec:RH}

We turn now to our main object: the homology of the normalized complex of simplicial chains in the nerve of a directed graph.
The \emph{normalized simplicial chain complex} of a simplicial set $S$,
with coefficients in a ring $R$, is a chain complex $N(S)$ of $R$-modules
with basis in degree $k$ given by the non-degenerate $k$-simplices of $S$;
see~\cite[8.3.6 and 8.3.7]{Weibel}.
To emphasize that the normalized chains on the nerve of a graph carries information about the reachability relation, we call it the \emph{reachability complex} of a directed graph; Di \textit{et al} call it the \emph{complex of regular accessible sequences} \cite[Section 1.6]{DIMZ}. Explicitly, it can be described as follows.

\begin{definition}\label{def:BPC}
Fix a commutative ground ring $R$. Let $G$ be a directed graph. The \emph{reachability complex} of $G$ is the chain complex $\bpc_\ast(G)$ of \(R\)-modules which in degree \(k\) is freely generated by those tuples \((v_0, \ldots, v_k)\) of vertices in \(G\) such that
    \begin{itemize}
        \item adjacent entries are distinct: \(v_{j-1}\neq v_j\) for \(j=1,\ldots, k\), and
        \item for each pair of adjacent entries \(v_{j-1},v_j\) there exists a directed path from \(v_{j-1}\) to \(v_j\) in \(G\).
    \end{itemize}
    The differential $\partial$ of $\bpc_\ast(G)$ acts on a generator \((v_0, \ldots, v_k)\) by omitting each term from the tuple in turn, and taking the alternating sum:
    \[\partial(v_0,\ldots,v_k) = \sum_{j=0}^k (-1)^j (v_0,\ldots,\widehat{v_j},\ldots,v_k).\]
    Here any term in which adjacent entries coincide is omitted. 

    We denote the homology of the reachability complex by
    \[\bph_\ast(G) = H_\ast(\bpc_\ast(G))\]
    and refer to it as the \demph{reachability homology} of \(G\). When necessary for clarity, we will write $\bpc_\ast(G;R)$ and $\bph_\ast(G;R)$ to emphasize the ground ring.
\end{definition}

\begin{remark}[The magnitude-path spectral sequence]\label{rmk:MPSS}
    The chain complex $\bpc_\ast(G)$ admits a natural filtration,
    under which a generator $(x_0,\ldots,x_k)$ lies in the filtration
    given by its \emph{length}
    $\ell(x_0,\ldots,x_k) = d(x_0,x_1)+\cdots+d(x_{k-1},x_k)$.
    The resulting spectral sequence is the 
    \emph{magnitude-path spectral sequence}, 
    or~\emph{MPSS},
    that appeared in~\cite[Remark~8.7]{HepworthWillerton2017}; whose importance was conclusively demonstrated by
    Asao~\cite{Asao-path}; which acquired its present name in Di \textit{et al}~\cite{DIMZ}; and whose properties are studied in detail by the present authors in~\cite{HepworthRoff2024}.
    
    The $E^1$-term of the MPSS is precisely the magnitude homology
    of~$G$, while---as Asao proved---the $E^2$-term 
    contains the path homology of $G$ on its horizontal axis. In~\cite{HepworthRoff2024} it is shown that many of the good features of path homology, including the K\"unneth and Mayer--Vietoris theorems, extend to the entire \(E^2\)-term of the MPSS, earning it the name \emph{bigraded path homology}. Moreover, bigraded path homology is a strictly finer invariant than path homology~\cite[Corollary 8.3]{HepworthRoff2024}. Thus, the magnitude-path spectral sequence encompasses several distinct invariants of directed graphs and places them in a systematic relationship with each other.
    
    If the filtration on \(\bpc_\ast(G)\) is bounded, meaning that in each degree $d$
    the lengths of the generators of $\bpc_d(G)$ are bounded
    above,
    then by construction the 
    MPSS converges to the reachability homology
    $\bph_\ast(G)$~(see 5.2.5, 5.4.2 and 5.5.1 of~\cite{Weibel}).
    Boundedness of $\bpc_\ast(G)$ can be guaranteed, for example,
    if $G$ is finite, or more generally if there is $L\geq 0$
    for which every distance in $G$ is either infinite or
    bounded above by $L$.
\end{remark}

\begin{remark}\label{rmk:acyclic_poset}
If a directed graph \(G\) is acyclic---meaning it contains no directed cycles---then the reachability relation on \(V(G)\) is antisymmetric, so \(\Pre(G)\) is a poset. In this case, \(\bpc_\ast(G)\) coincides with the normalized complex of simplicial chains in the \emph{order complex} of \(\Pre(G)\): the abstract simplicial complex whose faces are finite chains \(v_0 \leq v_1 \leq \cdots \leq v_k\) in \(\Pre(G)\) \cite{Wachs}. The reachability homology of \(G\) is then the \emph{poset homology} of the reachability relation on \(G\).

In fact, every preorder \(P\) is equivalent as a category to the poset \(\overline{P}\) obtained by identifying elements \(u\) and \(v\) whenever \(u \geq v\) and \(v \geq u\); this is an instance of the fact that every small category is equivalent to its skeleton \cite[Proposition 4.14]{JoyOfCats}. That equivalence induces a homotopy equivalence \(\Nv(P) \simeq \Nv(\overline{P})\). It follows that for every directed graph \(G\), the reachability homology \(\bph_\ast(G)\) coincides with the the poset homology of \(\overline{\Pre(G)}\). However, passing to the quotient poset destroys information that is required for the construction of the magnitude-path spectral sequence (see \Cref{rmk:MPSS}). Thus, if one is interested in understanding the spectral sequence, it is necessary to work with the complex \(\bpc_\ast(G)\).
\end{remark}

An explicit description of the induced maps follows immediately from \Cref{prop:nerve_maps} and standard facts about the normalized chain complex of a simplicial set.

\begin{proposition}[Induced maps]
    Any map of directed graphs \(f: G \to H\) induces a chain map
    \[f_\ast\colon \bpc_\ast(G) \longrightarrow \bpc_\ast(H) \]
    specified on generators by
    \[f_\ast(v_0,\ldots,v_k) =
        \begin{cases}
            (f(v_0),\ldots,f(v_k))
            &
            \text{if }f(v_{j-1})\neq f(v_j)
            \text{ for }j=1,\ldots,k
            \\
            0
            &
            \text{otherwise.}
        \end{cases}\]
    This in turn induces a map of homology groups
    \[
        f_\ast\colon
        \bph_\ast(G)
        \longrightarrow
        \bph_\ast(H). \tag*{\qed}
    \]
\end{proposition}

\begin{example}[Face graphs and Hasse diagrams]\label{eg:face_graphs_RH}
    To a simplicial complex $S$ we can associate two directed graphs:
    \begin{itemize}
        \item
        The \emph{face graph} $\calF_S$ of $S$
        is the graph whose vertices are the simplices
        of $S$, and in which there is a directed edge
        $\sigma\to\tau$ if and only if $\sigma$ is a face of $\tau$.
        In other words, it is obtained by taking the face poset of $S$
        and then applying the functor $\iota$.
        \item
        The \emph{Hasse diagram} $\calG_S$ of $S$
        is the graph whose vertices
        are again the simplices of $S$, and in which 
        there is a directed edge $\sigma\to\tau$ if and only if
        $\sigma$ is a codimension $1$ face of $\tau$.
        In other words, it is the Hasse diagram of the face poset
        of $S$.
    \end{itemize}
    Then the reachability homology of $\calF_S$ and $\calG_S$ 
    are both isomorphic to the simplicial homology of $S$:
    \[
        \bph_\ast(\calF_S)
        \cong
        \bph_\ast(\calG_S)
        \cong
        H_\ast(S)
    \]
    To see this, observe that 
    $\calF_S$ is the reflexive transitive closure of $\calG_S$; thus, by \Cref{rmk:trans_clos}, \(\Pre(\calF_S) = \Pre(\calG_S)\).
    It follows that $\bpc_\ast(\calF_S) = \bpc_\ast(\calG_S)$. 
    Indeed, in both cases, the generators
    $(\sigma_0,\ldots,\sigma_k)$ 
    of the reachability complex are the \emph{flags of simplices} 
    $\sigma_0\subset\cdots\subset\sigma_k$ 
    in $S$ (with strict inclusions).
    These flags are nothing but the simplices in
    the barycentric subdivision $\mathrm{sd}(S)$,
    and by inspecting the differentials we find that 
    $\bpc_\ast(\calF_S)$ and $\bpc_\ast(\calG_S)$
    can both be identified with
    the simplicial chains on $\mathrm{sd}(S)$.
    The claim now follows
    because $S$ and its subdivision have isomorphic
    homologies.
\end{example}

Recall that we say \emph{there is a long homotopy} from \(f: G\to H\) to \(g: G \to H\), 
and write \(f \rightsquigarrow g\), 
if for every \(x \in G\) there exists a path from \(f(x)\) to \(g(x)\) in \(H\) (\Cref{def:long_homotopy}). 
Recall also that $G$ and $H$ are said to
be \emph{long-homotopy equivalent} if there are maps $f\colon G\to H$ and $g\colon H\to G$ such that $f\circ g$ and $g\circ h$ are related to the relevant identity maps by zig-zags of long homotopies.

The next proposition follows immediately from \Cref{prop:long_homotopy} and standard facts about the homology of simplicial sets~\cite[Lemma~8.3.13]{Weibel}.

\begin{theorem}[Homotopy invariance]
\label{thm:long_htpy_inv}
    Let $f,g\colon G\to H$ be maps of directed graphs.
    If $f\rightsquigarrow g$, then the maps
    \[
        f_\ast,g_\ast\colon \bpc_\ast(G)
        \longrightarrow
        \bpc_\ast(H)
    \]
    are chain-homotopic, and consequently the maps
    \[
        f_\ast,g_\ast\colon\bph_\ast(G)
        \longrightarrow
        \bph_\ast(H)
    \]
    are equal.
    The relevant chain homotopy is given by the map
    \[
        s_\ast\colon
        \bpc_\ast(G)
        \longrightarrow
        \bpc_\ast(H)
    \]
    specified on generators by
    \[
        (x_0,\ldots,x_i)
        \longmapsto
        \sum_{j=0}^i
        (-1)^{j+1}
        (f(x_0),\ldots,f(x_j),g(x_j),\ldots,g(x_i))
    \]
    where, as usual, any tuple in which adjacent entries 
    coincide is omitted. 
    It follows that long-homotopy equivalences
    induce isomorphisms on reachability
    homology.
    \qed
\end{theorem}

\begin{example}[Three hexagons]\label{hexagons-one}
    Let us consider the three directed hexagons:
    
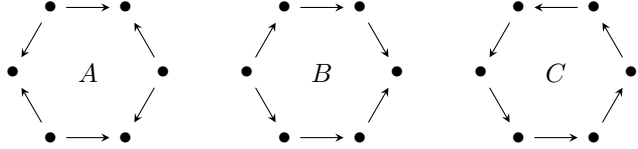
\[
        \begin{tikzpicture}
            \node (a) at (0:1) {$\bullet$};
            \node (b) at (60:1) {$\bullet$};
            \node (c) at (120:1) {$\bullet$};
            \node (d) at (180:1) {$\bullet$};
            \node (e) at (240:1) {$\bullet$};
            \node (f) at (300:1) {$\bullet$};
    
            \draw[stealth-] (b) -- (a);
            \draw[stealth-] (b) -- (c);
            \draw[stealth-] (d) -- (c);
            \draw[stealth-] (d) -- (e);
            \draw[stealth-] (f) -- (e);
            \draw[stealth-] (f) -- (a);
    
            \node () at (0:0) {$A$};
    
        \end{tikzpicture}
        \qquad
        \begin{tikzpicture}
            \node (a) at (0:1) {$\bullet$};
            \node (b) at (60:1) {$\bullet$};
            \node (c) at (120:1) {$\bullet$};
            \node (d) at (180:1) {$\bullet$};
            \node (e) at (240:1) {$\bullet$};
            \node (f) at (300:1) {$\bullet$};
    
            \draw[-stealth] (b) -- (a);
            \draw[-stealth] (c) -- (b);
            \draw[-stealth] (d) -- (c);
            \draw[-stealth] (d) -- (e);
            \draw[-stealth] (e) -- (f);
            \draw[-stealth] (f) -- (a);
    
            \node () at (0:0) {$B$};
    
        \end{tikzpicture}
        \qquad
        \begin{tikzpicture}
            \node (a) at (0:1) {$\bullet$};
            \node (b) at (60:1) {$\bullet$};
            \node (c) at (120:1) {$\bullet$};
            \node (d) at (180:1) {$\bullet$};
            \node (e) at (240:1) {$\bullet$};
            \node (f) at (300:1) {$\bullet$};
    
            \draw[-stealth] (a) -- (b);
            \draw[-stealth] (b) -- (c);
            \draw[-stealth] (c) -- (d);
            \draw[-stealth] (d) -- (e);
            \draw[-stealth] (e) -- (f);
            \draw[-stealth] (f) -- (a);
    
            \node () at (0:0) {$C$};
    
        \end{tikzpicture}
    \]
    Reachability homology distinguishes the first of these from
    the other two:
    \begin{center}
    \begin{tabular}{c | c c c c c }
        degree & 0 & 1 & 2 & 3 & $\cdots$
        \\
        \hline
        $\bph_\ast(A)$ & $R$ & $R$ & $0$ & $0$ & $\cdots$
        \\
        $\bph_\ast(B)$ & $R$ & $0$ & $0$ & $0$ & $\cdots$
        \\
        $\bph_\ast(C)$ & $R$ & $0$ & $0$ & $0$ & $\cdots$
    \end{tabular}
    \end{center}
    We can explain these outcomes as follows.
    Given a graph with a vertex $v$,
    we will write $c_v$ for the constant self-map
    that sends all vertices to $v$.
    \begin{itemize}
        \item
        Hexagon $A$ is equal to the face
        graph $\calF_{\partial\Delta^2}$
        of a triangle. 
        Its  reachability homology is therefore isomorphic
        to the simplicial homology of a
        triangle (\Cref{eg:face_graphs_RH}), which is in turn isomorphic
        to the ordinary homology of a
        circle.  This gives the computation
        of $\bph_\ast(A)$.
        \item
        In hexagon $B$, 
        there is a directed path 
        from any vertex to the rightmost
        vertex~$y$.
        This means that we have
        a long homotopy
        $\mathrm{Id}_B\rightsquigarrow c_y$.
        It follows that $B$ is long homotopy equivalent to a singleton graph,
        and therefore that its reachability
        homology coincides with that of a 
        singleton (\Cref{thm:long_htpy_inv}).
        This gives the computation of 
        $\bph_\ast(B)$.
        \item
        In hexagon $C$, there is a directed path from any
        vertex to any other, and consequently there is a 
        long homotopy $\mathrm{Id}_C\rightsquigarrow c_v$
        for any vertex $v$.
        It follows that, like $B$, the graph
        $C$ is long-homotopy equivalent to a singleton graph, and this gives
        the computation of $\bph_\ast(C)$.
    \end{itemize}
\end{example}

In \Cref{sec:RH_MV} we will prove an excision and a Mayer--Vietoris theorem concerning the reachability homology of a pair of directed graphs, defined as follows.

\begin{definition}\label{def:relative}
    Given any directed graph \(G\) and any subgraph \(A \subseteq G\), the reachability complex of \(\bpc_\ast(A)\) is a subcomplex of \(\bpc_\ast(G)\). We denote by \(\bpc_\ast(G,A)\) the quotient chain complex \(\bpc_\ast(G) / \bpc_\ast(A)\). The \emph{(relative) reachability homology of the pair} \((G,A)\) is \(\bph_\ast(G,A) = H_\ast(\bpc_\ast(G,A))\).
\end{definition}


\section{A K\"unneth theorem}\label{sec:RH_kunneth}

In \Cref{sec:digraph_preord} we described two binary operations on directed graphs that are natural to consider in the setting of this paper, namely the categorical product in \(\cn{DiGraph}\) and the product corresponding, under the shortest path metric, to the \(\ell_1\)-product of metric spaces. In the graph-theoretic literature they are known, respectively, as the strong product~$\squarediv$ and the box product~$\square$.

Both path homology and magnitude homology are known to satisfy a K\"unneth theorem with respect to the {box product}. (For path homology this is Theorem~4.7 in~\cite{PH-kunneth}, while for magnitude homology it is proved for undirected graphs as Theorem 5.3 in \cite{HepworthWillerton2017}, for classical metric spaces as Theorem 4.3 in \cite{BottinelliKaiser2021}, and for generalized metric spaces, including directed graphs, as Theorem 4.6 in \cite{Roff2023}.) It has been conjectured that path homology may also satisfy a K\"unneth theorem with respect to the {strong product} (\cite[\S 5.6]{GrigoryanSlides} and \cite[\S5.3]{TangYau2022}).

In \Cref{sec:digraph_preord} we saw that the reachability preorder does not distinguish these products: \Cref{prop:pre_digraph_monoidal} states that for every pair of directed graphs \(G\) and \(H\) we have
\[\Pre(G \squarediv H) \cong \Pre(G) \land \Pre(H) \cong \Pre(G \square H),\]
naturally in \(G\) and \(H\). It follows that if reachability homology satisfies a K\"unneth theorem with respect to either the strong or the box product, it must satisfy a K\"unneth theorem with respect to both. In fact that is the case, as we now prove.

As for the classical statement in singular homology, the proof has two parts: we first prove an Eilenberg--Zilber-type theorem for the reachability chain complex, and from there  K\"unneth theorem follows by general facts of homological algebra.

\begin{theorem}[Eilenberg--Zilber for the reachability complex]\label{thm:RH_EZ}
    Let \(G\) and \(H\) be directed graphs. There are chain homotopy equivalences
    \begin{equation}\label{eq:RH_EZ_strong}
        \bpc(G) \otimes \bpc(H) \simeq \bpc(G \squarediv H)
    \end{equation}
    and
    \begin{equation}\label{eq:RH_EZ_box}
        \bpc(G) \otimes \bpc(H) \simeq \bpc(G \square H),
    \end{equation}
    natural in \(G\) and \(H\).
\end{theorem}

\begin{proof}
    First, let \(P\) and \(P'\) be any preorders. In each dimension \(k\) there is a bijection
    \begin{align*}
    \Nv_k(P \land P') &= \{\left( (p_0, p_0'), \ldots (p_k, p_k')\right) \mid (p_i, p_i') \leq (p_{i+1}, p_{i+1}') \text{ in } P \land P'\} \\
    &= \{\left( (p_0, p_0'), \ldots (p_k, p_k')\right) \mid p_i \leq p_{i+1} \text{ in } P \text{ and } p_i' \leq p'_{i+1} \text{ in } P'\} \\
    &\cong \{ (p_0, \ldots, p_k) \mid p_i \leq p_{i+1} \text{ in } P\} \times \{ (p_0', \ldots, p_k') \mid p'_i \leq p'_{i+1}\text{ in } P'\} \\
    &= \Nv_k(P) \times \Nv_k(P').
    \end{align*}
    These bijections determine an isomorphism of simplicial sets
    \begin{equation}\label{eq:nv_iso}
        \Nv(P) \times \Nv(P') \cong \Nv(P \land P'),
    \end{equation}
    natural in \(P\) and \(P'\).

    Now, given a simplicial set \(S\), let \(N(S)\) denote the normalized complex of chains in \(S\). It is a standard fact of homological algebra that the Eilenberg--Zilber map reduces to a natural chain homotopy equivalence
    \[\triangledown: N(S) \otimes N(S') \xto{\sim} N(S \times S')\]
    for any pair of simplicial sets \(S\) and \(S'\). (A classical reference is Section 5 of \cite{EM1953}.) In our setting, combining this with \eqref{eq:nv_iso} gives a natural chain homotopy equivalence
    \begin{align*}
    \bpc(G) \otimes \bpc(H) &= N(\Nv(\Pre(G))) \otimes N(\Nv(\Pre(H))) \\
    &\xto{\triangledown} N(\Nv(\Pre(G)) \times \Nv(\Pre(H))) \\
    &\xto{\cong} N(\Nv(\Pre(G) \land \Pre(H))).
    \end{align*}
    In light of \Cref{prop:pre_digraph_monoidal}, this says there is a natural chain homotopy equivalence
    \[\triangledown: \bpc(G) \otimes \bpc(H) \to \bpc(G \square H) \cong \bpc(G \squarediv H).\]
    Explicitly, the map \(\triangledown\) is specified by
    \[
        \triangledown
        \left(
        (g_0,\ldots,g_p)
        \otimes
        (h_0,\ldots,h_q)
        \right)
        =
        \sum_\sigma
        \sign(\sigma)
        ((g_{i_0},h_{j_0}),\ldots,(g_{i_r},h_{j_r})) .
    \]
    Here $r=p+q$, and $\sigma$ runs over all sequences
    $((i_0,j_0),\ldots,(i_r,j_r))$ in which 
    $0\leq i_k\leq p$, $0\leq j_k\leq q$,
    and in which each term $(i_{k+1},j_{k+1})$ is obtained from its predecessor $(i_k,j_k)$ by increasing exactly one of
    the components by $1$.
    The coefficient $\sign(\sigma)$ is defined to be
    $(-1)^n$ where $n$ is the number of pairs $(i,j)$ for which
    $i=i_k\implies j< j_k$.
\end{proof}

\begin{theorem}[K\"unneth theorem for reachability homology]\label{thm:RH_kunneth}
Let \(R\) be a P.I.D. and let \(G\) and \(H\) be directed graphs. For each \(k\) there is a natural short exact sequence
\begin{align*}
0 \to \bigoplus_{i+j=k} \bph_i(G;R) \otimes &\bph_j(H;R) \to \bph_k(G \squarediv H;R) \\
& \to \bigoplus_{i+j=k-1} \mathrm{Tor}(\bph_i(G;R),\bph_j(H;R)) \to 0
\end{align*}
and a natural short exact sequence
\begin{align*}
0 \to \bigoplus_{i+j=k} \bph_i(G;R) \otimes &\bph_j(H;R) \to \bph_k(G \square H;R) \\
&\to \bigoplus_{i+j=k-1} \mathrm{Tor}(\bph_i(G;R),\bph_j(H;R)) \to 0.
\end{align*}
Both sequences split, but not naturally.
\end{theorem}

\begin{proof}
Since the reachability complex is freely generated, hence flat, the algebraic K\"unneth theorem tells us that there is for every \(k\) a natural short exact sequence
\begin{align*}
0 \to \bigoplus_{i+j=k} \bph_i(G;R) \otimes &\bph_j(H;R) \to H_k\left(\bpc(G;R) \otimes \bpc(H;R)\right) \\
&\to \bigoplus_{i+j=k-1} \mathrm{Tor}(\bph_i(G;R),\bph_j(H;R)) \to 0,
\end{align*}
which splits, though not naturally. (A statement and proof of the algebraic K\"unneth theorem can be found in~\cite[VI.3.3]{CartanEilenberg};
the statement there is for hereditary rings,
but this includes P.I.D.s as on~\cite[pp.13]{CartanEilenberg}.)

Meanwhile, \Cref{thm:RH_EZ} implies that there are natural isomorphisms
\[\bph_*(G \squarediv H) \cong H_*\left(\bpc(G;R) \otimes \bpc(H;R)\right) \cong \bph_*(G \square H).\]
Using the first of these isomorphisms to replace the middle term in the algebraic K\"unneth sequence yields the K\"unneth sequence for the strong product; using the second yields the K\"unneth sequence for the box product.
\end{proof}


\section{A Mayer--Vietoris sequence}\label{sec:RH_MV}

The magnitude homology of undirected graphs is known to satisfy an excision and a Mayer--Vietoris theorem with respect to a class of subgraph inclusions known as \emph{projecting decompositions} \cite[Theorem 6.6]{HepworthWillerton2017}. Recently, it has been shown by Carranza \emph{et al} that \(\cn{DiGraph}\) carries a cofibration category structure in which the weak equivalences are maps inducing isomorphisms on path homology, and the cofibrations are projecting decompositions satisfying a certain additional condition \cite[Definition 2.8]{CDKOSW}. In particular, they prove that the path homology of directed graphs satisfies an excision theorem with respect to these cofibrations \cite[\S 3]{CDKOSW}. In~\cite[\S 6]{HepworthRoff2024} it is proved that this excision theorem (and its associated Mayer--Vietoris theorem) extends to bigraded path homology, and that the cofibration category structure in \cite{CDKOSW} has a natural refinement, for which the cofibrations are the same but the weak equivalences are maps inducing isomorphisms on bigraded path homology~\cite[\S 7]{HepworthRoff2024}.

Here we prove an excision theorem, and associated Mayer--Vietoris theorem, for reachability homology. We begin by introducing the `good pairs' of directed graphs for which the main results will hold; these, we call \emph{long cofibrations} (\Cref{def:long_cofib}). The definition of a long cofibration is a relaxation of the notion of cofibration introduced in~\cite{CDKOSW} and studied in~\cite{HepworthRoff2024}. That definition is phrased in terms of minimal paths; ours is obtained by considering only the \emph{existence} of paths, and not their lengths. For convenience we also dualize things, so that if \(A \hookrightarrow G\) is a cofibration in the sense of \cite{CDKOSW}, then corresponding map between the transpose graphs---in which the direction of every edge has been reversed---is a long cofibration. This dualization allows us to give an alternative characterization of long cofibrations, as precisely those subgraph inclusions which induce \emph{Dwyer morphisms} upon passing from directed graphs to preorders by taking the reachability relation (\Cref{prop:cofib_dwyer}).

Dwyer morphisms are a special class of functors which played a central role in the development of the homotopy theory of small categories due to their convenient properties with respect to pushouts (the classical account of this is Thomason \cite{Thomason1980}). We are able to exploit those properties to prove a homotopy-theoretic excision theorem for long cofibrations (\Cref{thm:excision_nerve}), from which our Mayer--Vietoris theorem follows (\Cref{thm:excision_MV_reach}).

Throughout this section the term `path' should always be understood to mean `directed path'.

\begin{definition}
    Let \(X\) be a directed graph and \(A \subseteq X\) a subgraph. The \demph{reach} of \(A\), denoted \(rA\), is the induced subgraph of \(X\) on the set of all vertices which admit a path from some vertex of \(A\).
\end{definition}

\begin{definition}\label{def:long_cofib}
    A \demph{long cofibration} of directed graphs is an induced subgraph inclusion \(A \hookrightarrow X\) for which the following conditions hold:
    \begin{enumerate}
        \item There are no paths in \(X\) from vertices not in \(A\) to vertices in \(A\).\label{cond:long_cofib1}
        \item For each \(x \in rA\) there is a vertex \(\pi(x) \in A\) with the property that, for each \(a \in A\), the vertex \(a\) admits a path to \(x\) if and only if it admits a path to \(\pi(x)\). \label{cond:long_cofib2}
    \end{enumerate}
\end{definition}

Before stating the theorem, let us make a few remarks on the definition. First, condition (\ref{cond:long_cofib1}) is equivalent to requiring that there are no \emph{edges} from vertices not in \(A\) to vertices in \(A\). It implies that the subgraph \(A\) possesses a weak sort of `convexity': if \(a,b \in A\) are such that \(a\) admits a path to \(b\) in \(X\), then \(a\) admits a path to \(b\) in \(A\). In terms of the shortest path metric, this amounts to saying that for \(a,b \in A\) we have \( d_X(a,b) < \infty\) if and only if \(d_A(a,b) < \infty\).

Second, by taking \(a=\pi(x)\) in condition (\ref{cond:long_cofib2}), we see that for each \(x \in rA\) there is a path from \(\pi(x)\) to \(x\). If \(x\) is in \(A\) then, taking \(a = x\) in condition (\ref{cond:long_cofib2}), we see that there is also a path from \(x\) to \(\pi(x)\).

Third, the assignment $x\mapsto \pi(x)$  is not uniquely determined: each $\pi(x)$ can be replaced by any other vertex that it can reach and that can reach it. In particular, the previous remark implies that we are free to choose \(\pi(a) = a\) for every \(a \in A\). Finally, there is no guarantee that $\pi$ forms a map of directed graphs, or indeed that it \emph{can} form a map of directed graphs.

\begin{remark}   
    The use of the terminology `long cofibration' is not intended to indicate (at this stage) that this class of maps itself
    forms the class of cofibrations in a homotopy-theoretical 
    framework such as a model category structure.
    Instead, we intend to emphasize that the class of long cofibrations contains all cofibrations in the cofibration category structures exhibited in~\cite{CDKOSW} and \cite{HepworthRoff2024}.
\end{remark}

The main theorem in this section runs as follows. To state it, we note that the category \(\cn{DiGraph}\) admits all small limits and colimits, including in particular all pushouts. (This is explained, for instance, in the paragraph preceding Lemma 1.10 in \cite{CDKOSW}.)

\begin{theorem}[Excision and Mayer--Vietoris]\label{thm:excision_MV_reach}
    Let \(i: A \to X\) be a long cofibration, and \(f: A \to Y\) any map of directed graphs. Consider their pushout in \(\cn{DiGraph}\):
    \[\xymatrix{
        A\ar[r]^f\ar[d]_i & Y\ar[d]^j
        \\
        X\ar[r]_-g & X\cup_A Y
    }\]
    We have \emph{excision} on reachability homology, meaning that the map
    \[
        g_\ast\colon \bph_\ast(X,A)
        \xrightarrow{\ \cong\ }
        \bph_\ast(X\cup_A Y,Y) 
    \]
    is an isomorphism. We also have the \emph{Mayer-Vietoris sequence}, meaning that there is a long exact sequence:
    \[
        \cdots
        \to
        \bph_\ast(A)
        \xrightarrow{(f_\ast,-i_\ast)}
        \bph_\ast(X)\oplus\bph_\ast(Y)
        \xrightarrow{g_\ast\oplus j_\ast}
        \bph_\ast (X\cup_A Y)
        \xrightarrow{\partial_\ast}
        \bph_{\ast-1}(A)
        \to
        \cdots
    \]
\end{theorem}

The proof of \Cref{thm:excision_MV_reach} will occupy the remainder of the section. First, though, we return to the example of simplicial
complexes.

\begin{example}\label{ex:simplicial-cofibration}
    Let $Q$ be a simplicial complex, and let $P\subseteq Q$ be a
    subcomplex.
    Then the induced maps of face graphs and Hasse diagrams
    \[
        \calF_P\to\calF_Q,
        \qquad
        \calG_P\to\calG_Q
    \]
    (see Example~\ref{eg:face_graphs_RH})
    are long cofibrations if and only if $P$ is a \emph{full}
    subcomplex, or in other words, if and only if a simplex
    of $Q$ lies in $P$ whenever all of its vertices do,
    as we now explain.
    We deal with the case of $\calG_{(-)}$ for simplicity,
    but the explanation for $\calF_{(-)}$ is identical.
    
    Recall that if $S$ is a simplicial complex
    then the vertices of $\calG_S$ are simply the simplices of $S$,
    and observe that in $\calG_S$ there is a path from $\sigma$ to 
    $\tau$ if and only if $\sigma\subseteq\tau$.
    With this interpretation in hand, 
    the conditions for the map $\calF_P\to\calF_Q$
    to be a long cofibration then amount to the following:
    \begin{itemize} 
        \item
        Condition~\eqref{cond:long_cofib1} states that,
        if $\tau$ is a simplex of $P$ 
        and $\sigma$ is a simplex not in $P$, 
        then $\sigma$ is not a face of $\tau$. 

        \item
        Observe that  the reach $r\calG_P\subseteq\calG_Q$ 
        consists of all simplices $\sigma$ of $Q$
        that contain at least one vertex of $P$.
        Condition~\eqref{cond:long_cofib2} therefore
        requires, for each such $\sigma$, the existence of a simplex $\pi(\sigma)$
        of $P$ with the property that a simplex $\tau$ of $P$
        is contained in $\sigma$ 
        if and only if it is contained in $\pi(\sigma)$.
    \end{itemize}
    The first condition holds in all cases 
    by the very definition of subcomplex.
    The second condition requires $P$ to be a full subcomplex,
    as one sees by considering a simplex of $Q$ whose vertices all 
    lie in $P$. 
    And when $P$ is a full subcomplex, 
    the second condition holds by defining $\pi(\sigma)=P\cap\sigma$ for $\sigma\in r\calG_P$.
    
    Suppose now that $X$ is a simplicial complex with
    subcomplexes $A,B\subseteq X$, such that
    $X=A\cup B$. 
    Then we obtain the following pushout diagrams of graphs:
    \[
    \xymatrix{
        \calF_{A\cap B}\ar[r]\ar[d] & \calF_B\ar[d]
        \\
        \calF_A\ar[r] & \calF_X
    }
    \qquad\qquad
    \xymatrix{
        \calG_{A\cap B}\ar[r]\ar[d] & \calG_B\ar[d]
        \\
        \calG_A\ar[r] & \calG_X
    }
    \]
    Assuming that $A\cap B$ is full in $A$,
    the two pushout squares then give us excision and
    Mayer--Vietoris theorems in reachability homology.
    However, as in~\Cref{eg:face_graphs_RH}, the reachability homology
    of the face graph and Hasse diagram of a simplicial complex
    is precisely the simplicial homology of that complex,
    so that in this case~\Cref{thm:excision_MV_reach} recovers the 
    excision and Mayer--Vietoris theorems for simplicial homology.
\end{example}

We turn now to the proof of \Cref{thm:excision_MV_reach}. The proof rests on the fact---established below in \Cref{prop:cofib_dwyer}---that a long cofibration is precisely a map of directed graphs that induces, on taking the reachability relation (and passing from preorders to categories), a {Dwyer morphism}. A Dwyer morphism is a subcategory inclusion with properties analogous to those of the inclusion of a deformation retract of spaces. Since we will only deal directly with Dwyer morphisms between preorders, we spell out the definition only in that special case. For the general definition, see \cite[Definition 4.1]{Thomason1980} and \cite[D\'efinition 1]{Cisinski}. (In the case of preorders, these two definitions coincide.)

Recall that \emph{sub-preorder} of a preorder \(Q\) is a subset \(P \subseteq Q\) with preorder inherited from \(Q\). A monotone map \(P \hookrightarrow Q\) is the inclusion of a sub-preorder if and only if, as a functor, it is fully faithful and injective on objects---in other words, the inclusion of a full subcategory.

\begin{definition}\label{def:dwyer}
A \demph{Dwyer morphism} of preorders is the inclusion of a sub-preorder \(P \hookrightarrow Q\) which factors through another sub-preorder \(U \hookrightarrow Q\) where
\begin{enumerate}
    \item whenever \(x \in Q\) and \(y \in P\) are such that \(x \leq y\), then \(x \in P\); \label{cond:dwyer1}
    \item whenever \(x \in Q\) and \(u \in U\) are such that \(u \leq x\), then \(x \in U\); \label{cond:dwyer2}
    \item there is a monotone retraction \(p: U \to P\) such that \(p(u) \leq u\) for all \(u \in U\).\label{cond:dwyer3}
\end{enumerate}
\end{definition}

\begin{proposition}\label{prop:cofib_dwyer}
    A map of directed graphs \(f: A \to X\) is a long cofibration if and only if \(\Pre(f): \Pre(A) \to \Pre(X)\) is a Dwyer morphism.
\end{proposition}

\begin{proof}
    Observe that \(f\) satisfies condition (\ref{cond:long_cofib1}) in the definition of a long cofibration if and only if \(\Pre(f)\) satisfies condition (\ref{cond:dwyer1}) in the definition of a Dwyer morphism. The task is to show that conditions (\ref{cond:dwyer2}) and (\ref{cond:dwyer3}) in \Cref{def:dwyer} are together equivalent to condition (\ref{cond:long_cofib2}) in \Cref{def:long_cofib}.
    
    Suppose \(f\) is a long cofibration. Condition (\ref{cond:dwyer2}) for Dwyer morphisms is satisfied by \(U = \Pre(rA)\), and according to condition (\ref{cond:long_cofib2}) for long cofibrations (and the remarks following \Cref{def:long_cofib}) we can choose a function \(\pi: V(rA) \to V(A)\) such that \(\pi(a) = a\) for all \(a \in V(A)\). This function may not be a map of directed graphs---yet it will induce a monotone map \(\Pre(\pi): \Pre(rA) \to \Pre(A)\) satisfying condition (\ref{pt:dwyer3}) for Dwyer morphisms. Indeed, we have seen that for each \(x \in V(rA)\) the vertex \(\pi(x)\) admits a path to \(x\), which says that \(\pi(x) \leq x\) in \(\Pre(rA)\), so we just need to verify that \(\Pre(\pi)\) is monotone. Take \(x,y \in rA\)  such that \(x \leq y\) in \(\Pre(rA)\); that is, there exists a path from \(x\) to \(y\) in \(rA\). Since \(\pi(x)\) admits a path to \(x\), it admits a path to \(y\), and as \(\pi(x)\) belongs to \(A\), condition (\ref{cond:long_cofib2}) for long cofibrations says \(\pi(x)\) must admit a path to \(\pi(y)\). Thus, \(\pi(x) \leq \pi(y)\) in \(\Pre(A)\).

    Conversely, suppose \(\Pre(f)\) is a Dwyer morphism. Choose a factorization
    \[\Pre(A) \hookrightarrow U \hookrightarrow \Pre(X)\]
    and a retraction \(p: U \to \Pre(A)\) satisfying conditions (\ref{cond:dwyer2}) and (\ref{cond:dwyer3}) for Dwyer morphisms. Condition (\ref{cond:dwyer2}) 
    says there are no paths in \(X\) from vertices belonging to \(U\) to vertices outside \(U\). As \(U\) contains all vertices in \(A\), it must contain all vertices that can be reached in \(X\) by paths from \(A\); thus, \(rA \subseteq U\).  We claim that \(\pi = p|_{rA}\) satisfies condition (\ref{cond:long_cofib2}) for long cofibrations. For this, we need to see that for each \(x \in V(rA)\) and \(a \in V(A)\), the vertex \(a\) admits a path to \(x\) if and only if it admits a path to \(\pi(x)\). Take \(x \in V(rA)\) and \(a \in A\). That \(\pi(x) \leq x\) in \(U\) says \(\pi(x)\) admits a path to \(x\), so, given a path from \(a\) to \(\pi(x)\), we can construct one from \(a\) to \(x\). On the other hand, given a path from \(a\) to \(x\) we have \(a \leq x\) in \(\Pre(rA)\), and monotonicity of \(p\) gives \(\pi(a) \leq \pi(x)\); that is, there exists a path from \(a = \pi(a)\) to \(\pi(x)\).
\end{proof}

The importance of Dwyer morphisms in the homotopy theory of small categories derives from their good properties with respect to pushouts. The next statement summarizes the facts that are relevant to our story.

\begin{theorem}\label{thm:dwyer_props}
Let \(A, X\) and \(Y\) be preorders, \(i: A \hookrightarrow X\) a Dwyer morphism, and \(f: A \to Y\) any monotone map. Consider their pushout in \(\cn{Cat}\):
\[
\begin{tikzcd}
    A \arrow{r}{f} \arrow[hookrightarrow,swap]{d}{i} & Y \arrow{d}{j} \\
    X \arrow[swap]{r}{g} & X \cup_A Y
\end{tikzcd}
\]
The following facts hold:
\begin{enumerate}
        \item The category \(X \cup_A Y\) is a preorder. \label{pt:dwyer1}
        \item The map \(j: Y \to X \cup_A Y\) is a Dwyer morphism.  \label{pt:dwyer2}
        \item The natural map \(\Nv(X) \cup_{\Nv(A)} \Nv(Y) \to \Nv(X \cup_A Y)\) is a weak equivalence.  \label{pt:dwyer3}
    \end{enumerate}
\end{theorem}

\begin{proof}
    Thomason proves in \cite[Lemma 5.6, statement 4]{Thomason1980} that if \(A\), \(X\) and \(Y\) are posets and \(i\) is a Dwyer morphism, then \(X \cup_A Y\) is a poset; the argument he outlines also suffices to prove our statement (\ref{pt:dwyer1}). Statements (\ref{pt:dwyer2}) and (\ref{pt:dwyer3}) follow immediately from Proposition 4.3 in the same paper.
\end{proof}

Both the excision theorem and the Mayer--Vietoris sequence for reachability homology are derived from the following homotopy-theoretic version of excision, which is in turn an immediate consequence of the properties of Dwyer morphisms given in \Cref{thm:dwyer_props}.

\begin{theorem}\label{thm:excision_nerve}
    Let $i\colon A\to X$ be a long cofibration and $f\colon A\to Y$ any map of directed graphs, and consider their pushout in \(\cn{DiGraph}\):
    \begin{equation}\label{eq:excision_nerve}
\begin{tikzcd}
    A \arrow{r}{f} \arrow[hookrightarrow,swap]{d}{i} & Y \arrow{d}{j} \\
    X \arrow[swap]{r}{g} & X \cup_A Y
\end{tikzcd}
    \end{equation}
    The map \(j\) is a long cofibration, and the natural map 
    \[\Nv(X)\cup_{\Nv(A)}\Nv(Y) \to\Nv(X\cup_A Y)\]
    is a weak equivalence.
\end{theorem}

\begin{proof}
    The functor \(\Pre: \cn{DiGraph} \to \cn{PreOrd}\) is a left adjoint (\Cref{lem:Pre_adjoint}) and thus preserves pushouts, so upon applying \(\Pre\) to (\ref{eq:excision_nerve}) we obtain a pushout square in \(\cn{PreOrd}\) whose left leg, by \Cref{prop:cofib_dwyer}, is a Dwyer morphism.    Part (\ref{pt:dwyer1}) of \Cref{thm:dwyer_props} tells us that this pushout is preserved again upon passing from \(\cn{PreOrd}\) to \(\cn{Cat}\), and part (\ref{pt:dwyer2}) of the same theorem says that the map \(\Pre(j)\) is a Dwyer morphism.  By \Cref{prop:cofib_dwyer}, then, \(j\) must be a long cofibration. Finally, part (\ref{pt:dwyer3}) of \Cref{thm:dwyer_props} tells us that the natural map
    \[\Nv(X)\cup_{\Nv(A)}\Nv(Y) \to\Nv(X\cup_A Y)\]
    is a weak equivalence.
\end{proof}

We can now complete the proof of \Cref{thm:excision_MV_reach}.

\begin{proof}[Proof of \Cref{thm:excision_MV_reach}]
    Theorem~\ref{thm:excision_nerve} says we have
    a weak equivalence 
    $\Nv(X)\cup_{\Nv(A)}\Nv(Y) \to\Nv(X\cup_A Y)$
    and therefore, by taking normalized chains,
    a quasi-isomorphism:
    \begin{equation}\label{eq:excision_MV_reach1}
        \bpc_\ast(X)\oplus_{\bpc_\ast(A)}\bpc_\ast(Y) 
        \xrightarrow{\ \simeq\ }\bpc_\ast(X\cup_A Y)
    \end{equation}
    Here 
    $\bpc_\ast(X)\oplus_{\bpc_\ast(A)}\bpc_\ast(Y)$
    denotes the relevant pushout of chain complexes.
    This can be described explicitly as the quotient of
    $\bpc_\ast(X)\oplus\bpc_\ast(Y)$
    by the image of the map 
    $ 
        (f_\ast,-i_\ast)
        \colon
        \bpc_\ast(A)
        \to
        \bpc_\ast(X)\oplus\bpc_\ast(Y).
    $
    
    To prove excision, observe that we can now
    take the quasi-isomorphism~\eqref{eq:excision_MV_reach1} 
    and quotient by $\bpc_\ast(Y)$ to obtain the map
    $\bpc_\ast(X)/\bpc_\ast(A) 
    \to
    \bpc_\ast(X\cup_A Y)/\bpc_\ast(Y)$,
    or in other words the map
    $g_\ast \colon \bpc_\ast(X,A) \to \bpc_\ast(X\cup_A Y,Y)$,
    which is then a quasi-isomorphism by the five lemma.
    
    To prove the existence of the 
    Mayer-Vietoris sequence, observe that
    $\bpc_\ast(X)\oplus_{\bpc_\ast(A)}\bpc_\ast(Y)$
    fits into a short exact sequence of chain complexes
    \[\xymatrix{
        0
        \ar[r]
        &
        \bpc_\ast(A) 
        \ar[r]^-{(f_\ast,-i_\ast)}
        & 
        \bpc_\ast(X)\oplus\bpc_\ast(Y)
        \ar[r]
        \ar[dr]_-{g_\ast\oplus j_\ast}
        &
        \bpc_\ast(X)\oplus_{\bpc_\ast(A)}\bpc_\ast(Y)
        \ar[r]
        \ar[d]^\simeq
        &
        0
        \\
        &&&
        \bpc_\ast(X\cup_A Y)
        &
    }\]
    whose third term can be identified up to quasi-isomorphism
    using~\eqref{eq:excision_MV_reach1}.
    In the long exact sequence of homology groups
    associated to the short exact sequence,
    we may use the quasi-isomorphism to replace the homology of  
    $\bpc_\ast(X)\oplus_{\bpc_\ast(A)}\bpc_\ast(Y)$
    with that of 
    $\bpc_\ast(X\cup_A Y)$,
    giving the required result.
\end{proof}

\bibliographystyle{plain} 
\bibliography{2404_reachabilityhomology_arXiv} 

\end{document}